\documentclass[11pt,a4paper]{amsart}
\usepackage{enumerate}
\usepackage{pb-diagram}
\usepackage{microtype}
\usepackage{amsmath}
\usepackage{amssymb, pb-diagram}
\usepackage[all]{xy}
\usepackage{graphicx} 
\usepackage[dvipsnames]{xcolor}
\usepackage{chapterbib}
\usepackage{epsfig}
\usepackage{amsfonts}
\usepackage{amssymb}
\usepackage{amsmath}   
\usepackage{amsthm}
\usepackage{color}
\usepackage{dsfont}
\usepackage{latexsym}
\usepackage{hyperref}
\usepackage{mathabx}
\usepackage{graphicx}
\usepackage{tikz-cd}
\usepackage{mathabx}
\usepackage{enumitem}
\usepackage{url}

\input xy
\xyoption{all}
\usepackage{pgf, tikz}
\usetikzlibrary{arrows,shapes}
\usepackage{subcaption}

\renewcommand{\epsilon}{\varepsilon}

\renewcommand{\emptyset}{\varnothing}

\numberwithin{equation}{section}
\newtheorem{theorem}[equation]{Theorem}
\newtheorem{proposition}[equation]{Proposition}
\newtheorem{corollary}[equation]{Corollary}
\newtheorem{lemma}[equation]{Lemma}

\theoremstyle{definition}

\newtheorem{definition}[equation]{Definition}

\newtheorem{question}[equation]{Question}

\theoremstyle{remark}
\newtheorem{remark}[equation]{Remark}

\newcommand{\cohom}[3]{H^{{\raise1pt\hbox{$\scriptstyle#1$}}}(#2\>\!,#3)}
\newcommand{\tatecohom}[3]%
  {\widehat H^{{\raise1pt\hbox{$\scriptstyle#1$}}}(#2\>\!,#3)}

\newcommand{\Cohom}[3]%
  {H^{{\raise1pt\hbox{$\scriptstyle#1$}}}\big(#2\>\!,#3\big)}
\newcommand{\Tatecohom}[3]%
  {\widehat H^{{\raise1pt\hbox{$\scriptstyle#1$}}}\big(#2\>\!,#3\big)}

\newcommand{\homol}[3]{H_{{\lower1pt\hbox{$\scriptstyle#1$}}}(#2\>\!,#3)}
\newcommand{\homolog}[2]{H_{{\lower1pt\hbox{$\scriptstyle#1$}}}(#2)}

\newcommand{\frakF}{\mathfrak{F}}
\newcommand{\frakG}{\mathfrak{G}}

\catcode`\@=11

\newcommand{\calO}{\mathcal O}

\newcommand{\ModOFG}{\mathop{{\operator@font
Mod\text{-}}\calO_{\frakF}G}}
\newcommand{\OFGMod}{\mathop{\calO_{\frakF}G\text{-}{\operator@font
Mod}}}
\newcommand{\ModOGG}{\mathop{{\operator@font
Mod\text{-}}\calO_{\frakG}G}}
\newcommand{\OGGMod}{\mathop{\calO_{\frakG}G\text{-}{\operator@font
Mod}}}

\newcommand{\Fall}{\frakF_{\operator@font all}}
\newcommand{\Ffin}{\frakF_{\operator@font fin}}
\newcommand{\Fvc}{\frakF_{\operator@font vc}}
\newcommand{\Fic}{\frakF_{\operator@font ic}}
\newcommand{\Ffg}{\frakF_{\operator@font fg}}
\newcommand{\Fpc}{\frakF_{\operator@font pc}}
\newcommand{\Fab}{\frakF_{\operator@font ab}}
\newcommand{\Fvpc}{\frakF_{\operator@font vpc}}
\newcommand{\Fvab}{\frakF_{\operator@font vab}}

\catcode`\@=12

\DeclareMathOperator{\diam}{diam}

\renewcommand{\coprod}%
{\mathop{\rotatebox[origin=c]{180}{$\displaystyle\prod$}}\limits}

\newcommand{\mA}{\mathcal A}

\newcommand{\marc}{\mathcal{MA}(\Gamma)}
\newcommand{\ma}{\mathcal{MA}}

\newcommand{\mas}{\mathcal{A}}

\numberwithin{equation}{section}

\begin{document}

\title[Matching arc complexes]{Matching arc complexes: connectedness and hyperbolicity}


\author{Javier Aramayona}
\thanks{The research presented here is part of R.d.P.'s Master Thesis at UAM-ICMAT, and of A.F.'s work within the JAE-Intro programme at CSIC}

\address{Instituto de Ciencias Matem\'aticas (ICMAT). Madrid, Spain}
\email{javier.aramayona@icmat.es}

\author{Rodrigo de Pool}
\address{Universidad Aut\'onoma de Madrid \& ICMAT. Madrid, Spain}
\thanks{J.A. was supported by grant PGC2018-101179-B-I00. J.A. and R.d.P. acknowledge financial support from the Spanish Ministry of Science and Innovation, through the ``Severo Ochoa Programme for Centres of Excellence in R\&D (CEX2019-000904-S)''. A.F. was supported by JAE-Intro scholarship JAEINT20\_EX\_1062}

\author{Alejandro Fern\'andez}
\address{Universidad Aut\'onoma de Madrid \& ICMAT. Madrid, Spain}

\subjclass[2010]{}

\date{}

\keywords{}

\begin{abstract}
Addressing a question of Zaremsky, we give conditions on a finite simplicial graph which guarantee that the associated matching arc complex is connected and hyperbolic. 
\end{abstract}

\maketitle

\section{Introduction}

Arc and curve complexes have proved to be a powerful tool in the study of  mapping class groups. Of central importance in this direction is a celebrated theorem of Masur--Minsky \cite{MM99} which asserts that the curve complex is hyperbolic; see also \cite{A13,B14,HPW15,PS13}. The analogous statement for the arc complex was established by Masur-Schleimer \cite{MS13}; see also \cite{HH17,HPW15}.

In this note, we concentrate on a variation of the arc complex called the {\em matching arc complex}. Matching arc complexes, especially those associated to planar surfaces, have recently played a central role in the study of finiteness properties of {\em braided Thompson groups} \cite{Bux} and their relatives \cite{GLU20}. 

We now briefly introduce this complex, see Section \ref{sec:prelim} for details. Given a surface $S$ with marked points and a finite simplicial graph $\Gamma$, we fix an embedding $\Gamma \hookrightarrow S$ such that the image of every vertex of $\Gamma$ is a marked point of $S$. Say that an arc on $S$ is {\em compatible} with $\Gamma$ if there is an edge of $\Gamma$ with the same endpoints. The matching arc complex $\ma(S,\Gamma)$ is the simplicial complex whose $k$-simplices are sets of $k+1$ (isotopy classes of) compatible arcs on $S$ which are pairwise disjoint, including at their endpoints.

 Zaremsky \cite{Z,Z2} asked the question of determining for which graphs $\Gamma$ is the associated matching arc complex $\ma(S,\Gamma)$ connected and hyperbolic; the aim of this short note is to deal with this problem. 
 
 Our first result characterizes those graphs $\Gamma$ for which $\ma(S,\Gamma)$ is connected. Let $\Gamma_0$ be the simplicial graph obtained from $\Gamma$ by deleting every isolated vertex of $\Gamma$; abusing notation, we will write $|\Gamma_0|$ for the number of vertices of $\Gamma_0$. We will prove: 

\begin{theorem}
Let $\Gamma$ be a finite simplicial graph with $|\Gamma_0|\ge 5$. Then $\ma(S,  \Gamma)$ is connected if and only if the complement of every edge of $\Gamma$ contains at least one edge. 
\label{main:conn}
\end{theorem}

We remark that, in the special case when $\Gamma$ is a {\em complete} or {\em linear} graph on $n\ge 5$ vertices,  Theorem \ref{main:conn} follows from Theorem 3.8 and Corollary 3.11 of \cite{Bux}. 

\medskip

Next, we give a condition on $\Gamma$ that ensures the hyperbolicity of the associated matching arc complex:

\begin{theorem}
Let $\Gamma$ be a finite simplicial graph for which $\ma(S,\Gamma)$ is connected. If $\Gamma_0$ is connected and contains a triangle, then $\ma(S,\Gamma)$ is hyperbolic.
\label{main:hyp}
\end{theorem}

We stress that if $\Gamma$ is the complete graph on at least five vertices, then Theorem \ref{main:hyp} follows from the combination of \cite[Theorem 5.2]{DFV18} and Lemma \ref{lem:qi} below.

\begin{remark}
The inclusion  of the matching arc complex into the arc complex is never a quasi-isometry, and hence Theorem \ref{main:hyp} does not follow (or not immediately, at least) from the fact that arc complexes are hyperbolic. 
\end{remark}

\begin{remark}
Not all finite simplicial graphs yield a hyperbolic matching arc complex. Indeed, the {\em Disjoint Holes Principle} of Masur--Schleimer \cite[Lemma 5.12]{MS13} implies that if  $\Gamma$ is a bipartite graph whose parts have at least four vertices each, then $\ma(S,\Gamma)$ contains a quasi-isometrically embedded copy of $\mathbb Z^2$, and therefore cannot be hyperbolic. In particular, matching arc complexes associated to trees or even-sided polygons  are never hyperbolic. As a special case,  matching arc complexes associated to linear graphs (one of the family of matching arc complexes studied in \cite{Bux}) are never hyperbolic. Of course, there are connected graphs that are triangle-free but not bipartite, and it remains an interesting problem to determine if the associated matching arc complexes are hyperbolic or not. 
\end{remark}

Our proof of Theorem \ref{main:hyp} adapts the {\em unicorn path} machinery of Hensel-Przytycki-Webb \cite{HPW15}, which they used to prove the {\em uniform hyperbolicity} of arc graphs (the 1-skeleta of  arc complexes). The difficulties we encounter boil down to the fact that a unicorn path between two vertices of $\ma(S,\Gamma)$ may contain edges that do not belong to $\ma(S,\Gamma)$; in order to overcome this problem,  we will use the structure of $\Gamma$, plus techniques similar in spirit to those of \cite[Section 5]{DFV18}. 

One final remark is that, in stark contrast to the situation in \cite{HPW15}, our arguments produce a hyperbolicity constant (for matching arc graphs) that depends on the diameter of $\Gamma$. In light of this, we ask: 

\begin{question}
Is the family of hyperbolic matching arc graphs {\em uniformly} hyperbolic? 
\end{question}

\noindent{\bf Acknowledgements.} We thank Matt Zaremsky for conversations, and for comments on an earlier draft.

\section{Preliminaries}
\label{sec:prelim}

Throughout this note, $S$ will be a connected orientable surface with empty boundary and negative Euler characteristic. Let $g$ and $n$ be, respectively, the genus and number of punctures of $S$; as usual, we will treat punctures either as topological ends or as marked points on $S$. 

\subsection{Arcs} An {\em arc} is the image of a continuous  map $\alpha:[0,1] \to S$ such that its restriction to $(0,1)$ is injective, $\alpha(0)$ and $\alpha(1)$ are marked points of $S$, and $\alpha((0,1))$ is disjoint from the set of marked points of $S$. We say that an arc is {\em essential} if it does not bound a disk on $S$. In order to keep notation light, we will refer to the isotopy class  (relative to endpoints) of an essential arc simply as an {\em arc}, and will blur the difference between arcs and their representatives. 

Given arcs $\alpha$ and $\beta$, their {\em geometric intersection number} $i(\alpha, \beta)$ is defined as the minimum number of intersection points (not counting endpoints) between representatives of $\alpha$ and $\beta$. We say that representatives of $\alpha$ and $\beta$ are in {\em minimal position} if they realize the intersection number between $\alpha$ and $\beta$. Note that, given any set of arcs, it is possible to find representatives in minimal position by endowing $S$ with a hyperbolic metric and choosing geodesic representatives.

We will say that two arcs $\alpha$ and $\beta$ are {\em disjoint} if $i(\alpha, \beta)= 0$. If $\alpha$ and $\beta$ are not disjoint, we  say they {\em intersect}. Finally, if $i(\alpha,\beta)=0$ and $\alpha$ and $\beta$ have no endpoints in common,  we say they are {\em completely disjoint}. 

\begin{definition}[Arc complex]
The {\em arc complex} $\mA(S)$ is the simplicial complex whose $k$-simplices are sets of $k+1$ (pairwise) distinct  and  disjoint arcs.
\end{definition}

\subsubsection{Unicorn paths} We now briefly recall the main features of the {\em unicorn arc} machinery of \cite{HPW15}. Let $\alpha$ and $\beta$ be arcs in minimal position, and choose endpoints $v$ and $u$ of $\alpha$ and $\beta$, respectively. Consider  subarcs $\alpha'\subset \alpha$ and $\beta'\subset \beta$ which have one  endpoint at $v$ and $u$, respectively, and the other a common endpoint $p\in \alpha\cap \beta$.  If $\alpha' \cup \beta'$ has no self-intersections, then it is called a \emph{unicorn arc} obtained from $\alpha^v$ and $\beta ^u$. Observe that $\alpha' \cup \beta'$ is automatically essential, since $\alpha$  and $\beta$ are in minimal position. Also, note that the isotopy class of an unicorn arc is completely determined by the point $p$; however not every such point defines an unicorn arc, as for some choices the subarcs $\alpha'$ and $\beta'$ may intersect each other. We define a total order for unicorns arcs obtained from $\alpha^v$ and $\beta ^u$, namely: 
\[
\alpha'\cup \beta' \leq \alpha''\cup \beta'' \iff \alpha'' \subset \alpha '.
\]
Let 
\[\mathcal{U}(\alpha^v,\beta^u):= \left (c_0=\alpha, c_1,\dots,c_{n-1},c_n=\beta\right ),\] be the ordered tuple of unicorn arcs defined by $\alpha^v$ and $\beta ^u$;  we call this tuple a \emph{unicorn path} between $\alpha$ and $\beta$. The following is Remark 3.4 of \cite{HPW15}: 
\begin{lemma}
Any two consecutive elements of  $\mathcal{U}(\alpha^v,\beta^u)$ are disjoint. 
\end{lemma}
As a consequence, $\mathcal{U}(\alpha^v,\beta^u)$ is  really a path in $\mathcal A(S)$. The next lemma is \cite[Lemma 3.3.]{HPW15}:

\begin{lemma}\label{lem:unicorn_triangle}
Let $\alpha^v$, $\beta^u$ and $\gamma^w$ based at $v$, $u$ and $w$, respectively. For every $c\in \mathcal{U}(\alpha^v,\beta^u)$ there is $c^*\in  \mathcal{U}(\alpha^v,\gamma^w)\cup  \mathcal{U}(\gamma^w,\beta^u)$ with $i(c,c^*)=0$.
\end{lemma}

\subsection{Matching arc complex}
Let $\Gamma$ be a finite simplicial graph. We fix an embedding $\Gamma \hookrightarrow S$ such that the image of every vertex of $\Gamma$ is a marked point of $S$; in particular, we are implicitly assuming that $S$ has at least as many punctures as $\Gamma$ has vertices. In what follows, we will blur the difference between $\Gamma$ and its image under this embedding. An arc $\alpha$ is {\em compatible} with $\Gamma$ if there is an edge $e$ of $\Gamma$ with the same endpoints; we say that $\alpha$ is compatible with  $e$.

\begin{definition}[Matching arc complex]
The {\em matching arc complex} $\ma(S,\Gamma)$ is the simplicial complex whose $k$-simplices are sets of $k+1$ arcs compatible with $\Gamma$ and which are pairwise completely disjoint. 
\end{definition}
Throughout, we will write $d(\cdot,\cdot)$ for the natural path-metric on the vertex set of $\ma(S,\Gamma)$.

\section{Connectivity of $\marc$}\label{sec:connection}
In this section we prove Theorem \ref{main:conn}. Before doing so, we need a couple of  technical lemmas that will be used to (uniformly) bound the distance between  consecutive vertices in a unicorn path.

\begin{lemma}\label{lem:cuna_disjunta}
Let $\Gamma$ as in Theorem \ref{main:conn}. If $\alpha$ and $\beta$ are disjoint arcs on $S$ with exactly one endpoint in common, then $d(\alpha, \beta)\leq 4$.
\end{lemma}

\begin{proof}
Let $e,f\in \Gamma$ be edges compatible with $\alpha$ and  $\beta$, respectively. If there exists an edge $e'' \in \Gamma$ such that $e''\cap (e\cup f) = \emptyset$, then we can choose an arc in $\ma(S,\Gamma)$ compatible with $e''$ and which is completely disjoint from $\alpha \cup \beta$. In particular, $d(\alpha, \beta)= 2$. 

Suppose now that no such edge exists. By hypothesis, there exists an edge $e'$ with $e \cap e'= \emptyset$; we may assume that $e'$ and $f$ have exactly one endpoint in common, for otherwise we are back in the situation above. Similarly, there exists another edge $f'$,  disjoint from $f$ and with exactly one endpoint in common with $e$. There are two cases to consider, depending on whether $e'$ and $f'$ are disjoint or not. 

Assume first that $e' \cap f' = \emptyset$. Then, we can choose completely disjoint arcs $\gamma_1, \gamma_2\in \ma(S,\Gamma)$, compatible with $e'$ and $f'$ respectively,  such that $i(\gamma_i,\alpha) = i(\gamma_i,\beta)=0$. Then the sequence $\alpha, \gamma_1,\gamma_2, \beta$ gives a path in $\ma(S,\Gamma)$ of length three between $\alpha$ and $\beta$. 

It remains to consider the case where $e' \cap f' \ne \emptyset$; as $\Gamma$ is simplicial, this means that $e'$ and $f'$ have exactly one endpoint in common. Since $|\Gamma_0| \ge 5$, there exists an edge $e''$ that shares at most one endpoint with $e\cup f\cup e' \cup f'$. Moreover, since we are not in any of the cases considered previously, we may assume that $e''$ shares exactly one endpoint with both  $e$ and $f$, but is completely disjoint from $e'$ and $f'$. Then we may choose pairwise disjoint arcs $\gamma_1,\gamma_2, \gamma_3$ such that $\gamma_1$ (resp. $\gamma_2,\gamma_3$) is compatible with $e'$ (resp. $e'', f'$), and $i(\gamma_i,\alpha)= i(\gamma_i,\beta)=0$. Thus, $\alpha, \gamma_1,\gamma_2,\gamma_3,\beta$ yield a path of length four in $\ma(S,\Gamma)$ between $\alpha$ and $\beta$.  
\end{proof}

Next, we deal with arcs sharing two endpoints:

\begin{lemma}\label{lem:mismos_endpoints}
Let $\Gamma$ as in Theorem \ref{main:conn}. If $\alpha, \beta \in \marc$ are disjoint arcs with the same endpoints, then $d(\alpha, \beta)\leq 6$. 
\end{lemma}

\begin{proof}
If the complement of $\alpha \cup \beta$ is connected, then there is an arc completely disjoint from both, as we are assuming $|\Gamma_0| \ge 
5$. Thus, assume $\alpha \cup \beta$ separate $S$. If any of the connected components of the complement of $\alpha \cup \beta$ contains a pair of punctures spanning an edge of $\Gamma$, then again there exists an arc completely disjoint from $\alpha$ and $\beta$, and we are done. 

Thus suppose we are not in any of the above situations.  Let $e$ be the edge of $\Gamma$ which $\alpha$ and $\beta$ are compatible with. By the hypothesis on $\Gamma$, there exists an edge $f_1$ of $\Gamma$ disjoint from $e$; similarly, as $|\Gamma_0| \ge 5$, there exists an edge $f_2$ which shares at most one endpoint with $e\cup f_1$. We may choose disjoint arcs $\gamma_1,\gamma_2 \in \ma(S,\Gamma)$ such that $\gamma_i$ is compatible with $f_i$, and $i(\alpha, \gamma_1) = i(\beta, \gamma_2) = 0$. At this point, there are two cases to consider. 

Suppose first  that $f_2\cap e= \emptyset$.  If the arcs $\gamma_1$ and $\gamma_2$ are completely disjoint, then the sequence $\alpha, \gamma_1, \gamma_2,\beta$ yields a path of length three in $\ma(S,\Gamma)$. Otherwise, $\gamma_1$ and $\gamma_2$ satisfy the hypotheses of Lemma \ref{lem:cuna_disjunta} and therefore are at distance at most four, which implies the desired result. 

If, on the other hand, $f_2$ shares one endpoint with $e$, then $d(\beta,\gamma_2)\le 4$, by Lemma \ref{lem:cuna_disjunta}. Since $d(\gamma_1,\alpha) =1$, we are done. \end{proof}

We are finally ready to prove our first main result:

\begin{proof}[Proof of Theorem \ref{main:conn}]
If there existed one edge whose complement contains no edges, then we would be able to find a vertex of $\ma(S,\Gamma)$ with no neighbours, as any other arc would share at least one endpoint with it. 

For the other implication, suppose that the complement of every edge contains at least one edge. Let $\alpha$ and $\beta$ be any arcs in $\ma(S,\Gamma)$. Choose an arc $\beta'$, disjoint from $\beta$ and with the same endpoints as $\alpha$. Appealing to Lemmas \ref{lem:cuna_disjunta} or \ref{lem:mismos_endpoints} if necessary, we know that $d(\beta, \beta') \le 6$. Now, the tuple  $\mathcal{U}(\alpha^{v_1}, \beta'^{v_2})$ defines a sequence of vertices in $\ma(S,\Gamma)$; by Lemma \ref{lem:mismos_endpoints}, any two consecutive arcs of this sequence are at distance at most $6$ in $\ma(S,\Gamma)$, and thus we are done. 
\end{proof}

\subsection{Graphs with few vertices} Before we end this section, we make some comments on the need for a lower bound on $|\Gamma_0|$. First, observe that if $\ma(S,\Gamma)$ is connected then $|\Gamma_0|\ge 4$. Indeed,  as was the case in the proof of Theorem \ref{main:conn}, if this were not the case we would be able to find a vertex of $\ma(S,\Gamma)$ with no neighbours, since any other arc on $\ma(S,\Gamma)$ would share at least one endpoint with it. In other words, we have proved: 

\begin{corollary}
Let $\Gamma$ be a finite simplicial graph with $|\Gamma_0|\le 3$. Then $\ma(S,\Gamma)$ is not connected. 

\end{corollary}

Combining the above lemma with Theorem \ref{main:conn}, the only outstanding case to consider is  $|\Gamma_0| = 4$. We will prove: 

\begin{proposition}\label{prop:gamma_cuatro}
Let $\Gamma$ be a finite simplicial graph with $|\Gamma_0|= 4$. Then $\ma(S,\Gamma)$ is connected if and only $S$ has positive genus and $\Gamma_0$ has exactly two edges, which are disjoint.
\end{proposition}

\begin{proof}
Assume first that $\ma(S,\Gamma)$ is connected, and assume that $\Gamma_0$ does not satisfy the conclusion of the theorem. Since $|\Gamma_0| = 4$, the graph $\Gamma_0$ is one of the graphs depicted in Figure \ref{fig:graphs_up_to_isom}. For any of them, there are two edges $e_1,e_2$ of $\Gamma$ sharing an endpoint and such that any other edge in $\Gamma$ shares one endpoint with (at least) one of them. If $\alpha_1$ and $\alpha_2$ are arcs compatible with $e_1$ and $e_2$, respectively, there is no path between  $\alpha_1$ and $\alpha_2$ in $\ma(S,\Gamma)$.
 
\begin{figure}
    \centering
    \includegraphics[width=0.7\linewidth]{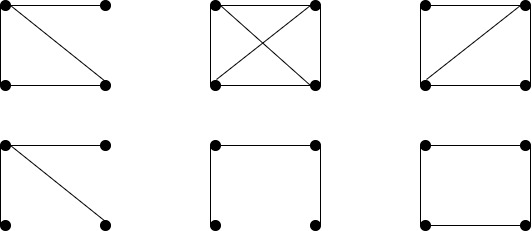}
    \caption{Combinatorial possibilities for a  connected graph with at most four vertices}
    \label{fig:graphs_up_to_isom}
\end{figure}

Therefore, $\Gamma$ has exactly two edges, which are disjoint.
Let $p,p'$ be two marked points of $S$ spanning an edge of $\Gamma$. If the genus of $S$ is zero, there are two disjoint arcs $\alpha,\beta$ on $S$, with the same endpoints, such that $p$ and $p'$ lie on different components of the complement of $\alpha \cup \beta$. We claim that there is no path in $\ma(S,\Gamma)$ between $\alpha$ and $\beta$. To see this, first observe that the Jordan Curve Theorem implies that $\alpha$ and $\beta$ are not at distance two. Suppose we had a path $\alpha=\gamma_0,\gamma_1, \ldots, \gamma_{2k}=\beta$  in $\ma(S,\Gamma)$. Since $\gamma_{2i}$ and $\gamma_{2i+2}$ are at distance $2$, then  $p$ and $p'$ cannot lie in different components of the complement of $\gamma_{2i} \cup \gamma_{2i+2}$. The same is true for $\gamma_{2i}$ and $\gamma_{2i+2k}$ for all $i$ and $k$; in particular, it holds for $\alpha$ and $\beta$, which is a contradiction. 

For the reverse implication, let $\alpha$ and $\beta$ be two arbitrary vertices of $\ma(S,\Gamma)$. Using a unicorn path argument implies that it suffices to check the case when  $\alpha$ and $\beta$ are disjoint and have the same endpoints. If $\alpha\cup \beta$ does not separate $S$, then $\alpha$ and $\beta$ are at distance two. Otherwise, they are at distance four in $\ma(S,\Gamma)$, see Figure \ref{fig:my_label}.

\begin{figure}
    \centering
    \includegraphics[width=0.7\linewidth]{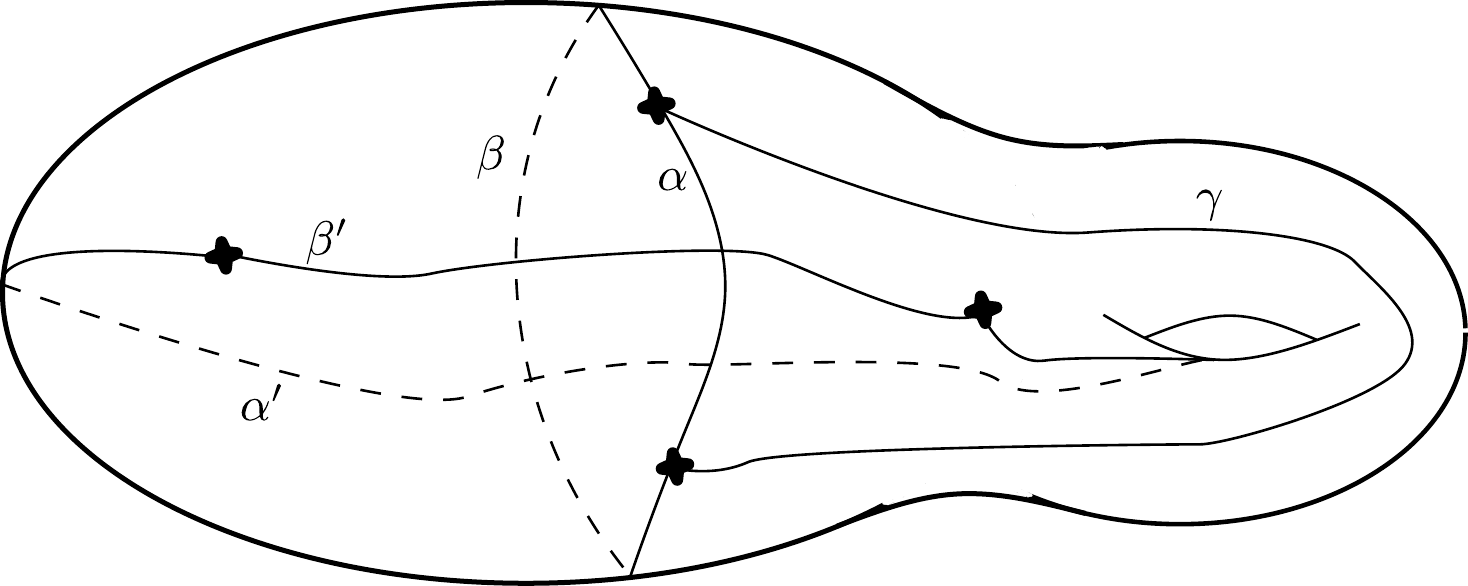}

    \caption{ The arcs $\alpha,\beta$ are connected via the path $\alpha, \alpha',\gamma,\beta', \beta$, on a surface of genus one with four marked points. The general situation is analogous.}
    \label{fig:my_label}
\end{figure}
\end{proof}

\section{Gromov hyperbolicity of $\marc$}
In this section we prove Theorem \ref{main:hyp}. We commence by passing to a related complex that is easier to analyze. Let $\mas(S,\Gamma)$ be the complex with the same vertex set as $\ma(S, \Gamma)$, but where simplices are spanned by  disjoint (but not necessarily completely disjoint) arcs; in other words, we are allowing arcs intersect in their endpoints. A direct consequence of Lemmas \ref{lem:cuna_disjunta} and  \ref{lem:mismos_endpoints} is the following: 

\begin{lemma}
If $\ma(S,\Gamma)$ is connected, then the inclusion map $\ma(S,\Gamma) \hookrightarrow \mas(S,\Gamma)$ is a quasi-isometry.
\label{lem:qi}
\end{lemma}

In light of the above lemma, it suffices to establish the hyperbolicity of $\mas(S,\Gamma)$. The main tool will be following result of  Masur--Schleimer\cite{MS13}; here, we use a similar form, as stated by Bowditch in \cite{B14}. Given a subset $C$ of a metric space $X$, denote by $N_k(C)$ the $k$-neighbourhood of $C$ in $X$, i.e., the set of points of $X$ at distance at most $k$ from $C$.

\begin{lemma}[Guessing geodesics lemma]
Let $G$ be a connected graph, and write $d$ for its usual distance function. Let $M>0$ be a constant, and suppose that for any vertices $x,y\in G$ there is a connected subgraph $A(x,y)$, with $x,y\in A(x,y)$, such that:
\begin{enumerate}
    \item If $d(x,y)\leq 1$, then $\diam(A(x,y))\le M$.
    \item For any vertices $x,y,z\in G$, $A(x,y) \subset N_M(A(x,z) \cup A(z,y))$.
\end{enumerate}
Then $G$ is  hyperbolic, with hyperbolicity constant that depends only on $M$.
\label{lem:guessing} 
\end{lemma}

In order to apply Lemma \ref{lem:guessing}, we use the unicorn path machinery of Hensel-Przytycki-Webb \cite{HPW15}. As mentioned in the introduction, the main hurdle in this direction is that the unicorn path between two vertices of $\mas(S,\Gamma)$ need not be contained in $\mas(S,\Gamma)$.



\subsection{A family of subpaths in $\mas(S,\Gamma)$}
From now on, we will assume that $\Gamma_0$ is a finite, connected simplicial graph that contains a triangle $\Delta$, and such that $\ma(S,\Gamma)$ is connected. Denote by $e_i$ the three edges of $\Delta$. To any arc $\alpha \in \mas(S,\Gamma)$ we associate, once and for all, three pairwise disjoint arcs $a_i$, called the {\em triangle arcs} associated to $\alpha$, such that: 
\begin{itemize}
    \item $a_i$ is compatible with $e_i$;
    \item $a_1 \cup a_2 \cup a_3$ bound a disk on $S$; 
    \item $i(\alpha,a_i)=0$ for $i=1,2,3$.
\end{itemize}
(As may have already become apparent, we are using Greek letters for arcs, and  Roman letters for their associated triangle arcs.) In the particular case when $\alpha$ is already compatible with some $e_i$, then we simply set $a_i:=\alpha$. Note that the existence of the arcs $a_i$ is easily verified using the Change of Coordinates Principle \cite[Section 1.3]{FM}.

Now, given two arcs $\alpha, \beta \in \mas(S,\Gamma)$, we define $A'(\alpha,\beta)$ to be the set of all those unicorn arcs defined by $a_i$ and $b_j$ that are elements of $\mas(S,\Gamma)$. Then, we define \[A(\alpha,\beta):=N_1(A'(\alpha,\beta)).\]
 Clearly, $A(\alpha, \beta)$ is connected and $\alpha,\beta \in A(\alpha,\beta)$. Now, we are left to prove that both conditions in Lemma \ref{lem:guessing} are satisfied. In what follows, we will write $v_1,v_2,v_3$ for the vertices of the triangle $\Delta$, where the edge $e_i$ has vertices $v_i$ and $v_{i+1}$. 

\begin{lemma}
For every $\alpha,\beta,\gamma \in \mas(S,\Gamma)$, \[A(\alpha,\beta)\subset N_2(A(\alpha,\gamma)\cup A(\gamma, \beta)).\]
\label{lem:thin}
\end{lemma}

\begin{proof}
It suffices to prove that $A'(\alpha,\beta)\subset N_1(A(\alpha,\gamma)\cup A(\gamma, \beta))$. To this end, let $\delta \in A'(\alpha, \beta)$; without loss of generality we may assume that $\delta \in \mathcal{U}(a_i^{v_1}, b_j^{v_2})$. By Lemma \ref{lem:unicorn_triangle}, there is a $
\delta^* \in \mathcal{U}(a_i^{v_1}, c_3^{v_3})     \cup   
\mathcal{U}(c_3^{v_3}, b_j^{v_2})
$ such that $d(\delta, \delta^*)=1$. Observe that $\delta^* \in A(\alpha,\gamma)\cup A(\gamma, \beta) $, which gives the result.
\end{proof}

\begin{lemma}
Let $\alpha, \beta \in \mas(S,\Gamma)$ with $d(\alpha, \beta) = 1$. Then \[\diam(A(\alpha,\beta)) \leq 2\diam(\Gamma_0)+9.\]
\label{lem:diam}
\end{lemma}

\begin{proof}
We will write $k= \diam(\Gamma_0)$, for compactness. By definition, it suffices to see that \[
\diam(A'(\alpha,\beta)) \leq 2k+7;
\]
for this, we will show that $d(\gamma, \{\alpha,\beta\}) \leq k+3$ for any $\gamma\in A'(\alpha,\beta)$.  Let $\gamma\in A'(\alpha,\beta)$, and recall that $a_i$ and $b_j$ are the triangle arcs of $\alpha$ and $\beta$, respectively. Up to renaming, we may assume that $ \gamma \in \mathcal{U}(a_m^{v_1}, b_n^{v_2})$. Let $\partial \alpha = \{x_1,x_2\}$ and $\partial \beta = \{y_1,y_2\}$. Choose a simple path \[
X:w_0=v_1, w_1=v_2, w_2, \dots, w_{k-1}, w_k=x_1
\]
in $\Gamma$; however, if $w_l \in \{y_1,y_2\}$ for some $l\le k$, then we truncate $X$ at $w_l$. 

Next, the triangle arcs $a_j$ bound a topological disk on $S$, and hence we may choose pairwise disjoint arcs $\alpha_i$, for $i=1,\ldots, k-1$, such that $\alpha_i$ is compatible with the edge spanned by $w_i$ and $w_{i+1}$, and \[i(\alpha_i, \alpha)= i(\alpha_i,a_j)= 0,\] for $j=1,2,3$. Analogously, we select pairwise disjoint arcs $\beta_1, \ldots, \beta_{k-1}$ such that $\beta_i$ is compatible with the edge spanned by $w_i$ and $w_{i+1}$, and \[i(\beta_i, \beta)= i(\beta_i,b_j)= 0,\] for $j=1,2,3$. Finally, we set $\alpha_0=a_m$ and $\beta_0=b_n$. 

\medskip

\noindent{\bf Claim.} Either $d(\gamma, \{\alpha, \beta\})\leq k+1$, or there exists $\gamma'\in \mathcal{U}(\alpha_{k-1}^{w_{k-1}}, \beta_{k-1}^{w_{k}})\cup \mathcal{U}(\alpha_{k-1}^{w_{k}}, \beta_{k-1}^{w_{k-1}})$ with $d(\gamma,\gamma')\leq k$. 

\medskip

We accept the claim as true for the time being and continue with our argument; we will establish the claim at the end of the proof. If we are in the first case of the claim, then we are done; hence, assume that we are in the second case. We distinguish two subcases: 

\medskip

\noindent{\bf{\em {Case 1:}}} $\gamma'\in \mathcal{U}(\alpha_{k-1}^{w_{k-1}}, \beta_{k-1}^{w_{k}})$.  We apply Lemma \ref{lem:unicorn_triangle}, obtaining $\gamma'' \in \mas(S,\Gamma)$, with $d(\gamma', \gamma'') \le 1$, such that \[\gamma''\in \mathcal{U}(\alpha_{k-1}^{w_{k-1}}, \alpha^{x_2})\cup \mathcal{U}(\alpha^{x_2}, \beta_{k-1}^{w_{k}}).\] Again, there are two cases to consider: 
\smallskip

{\em Subcase (1a):} Suppose first that $\gamma''\in \mathcal{U}(\alpha_{k-1}^{w_{k-1}}, \alpha^{x_2})$, and note that $ \mathcal{U}(\alpha_{k-1}^{w_{k-1}}, \alpha^{x_2}) = \{\alpha, \alpha_{k-1}\}$. As $d(\gamma, \gamma')\leq k$, $d(\gamma', \gamma'')\leq 1$ and $d(\gamma'', \alpha)\leq 1$, then $d(\gamma, \{\alpha,\beta\})\leq k+2$, and we are done.

\smallskip

{\em Subcase (1b):} Assume now that $\gamma'' \in \mathcal{U}(\alpha^{x_2}, \beta_{k-1}^{w_{k}})$. Lemma \ref{lem:unicorn_triangle} yields the existence of an arc  $\gamma''' \in \mas(S,\Gamma)$, with $d(\gamma'', \gamma''')\leq 1$, such that
$
\gamma''' \in \mathcal{U}(\alpha^{x_2}, \beta^v)\cup\mathcal{U}(\beta^v, \beta_{k-1}^{w_{k}}).
$
Note that $\mathcal{U}(\alpha^{x_2}, \beta^v)=\{\alpha, \beta\}$ and $\mathcal{U}(\beta^v, \beta_{k-1}^{w_{k}})=\{\beta, \beta_{k-1}\}$. In either situation, $d(\gamma''', \{\alpha,\beta\})\leq 1$ and hence $d(\gamma, \{\alpha,\beta\})\leq k+3$, as desired.

\medskip

\noindent{{\bf \em {Case 2:}}} $\gamma' \in \mathcal{U}(\alpha_{k-1}^{w_{k}}, \beta_{k-1}^{w_{k-1}})$.
Again, by Lemma \ref{lem:unicorn_triangle} there is $\gamma'' \in \mas(S,\Gamma)$,  with $d(\gamma', \gamma'')\leq 1$, such that \[
\gamma'' \in 
\mathcal{U}(\alpha_{k-1}^{w_{k}}, \alpha^{w_k})\cup \mathcal{U}(\alpha^{w_k}, \beta_{k-1}^{w_{k-1}})
,\] and we proceed as in the previous case. 


       


\medskip

At this point, all that remains is to prove the above Claim. 
To this end, we will apply an inductive argument based on the structure of $\Gamma$. First, since $\gamma \in \mathcal{U}(a_m^{v_1}, b_n^{v_2})$, Lemma \ref{lem:unicorn_triangle} implies that there is \[\gamma_1\in \mathcal{U}(\alpha_0^{w_0},\alpha_1^{w_2})\cup\mathcal{U}(\alpha_1^{w_2},\beta_0^{w_1})\] with $d(\gamma,\gamma_1)\leq 1$. If $\gamma_1\in \mathcal{U}(\alpha_0^{w_0},\alpha_1^{w_2})$ then 
$d(\gamma_1, \alpha)\leq 1$, as $\mathcal{U}(\alpha_0^{w_0},\alpha_1^{w_2})=\{\alpha_0,\alpha_1\}$; hence $d(\gamma, \{\alpha,\beta\})\leq 2$ and the result follows. If, on the contrary, $\gamma_1\in \mathcal{U}(\alpha_1^{w_2},\beta_0^{w_1})$, we continue this process inductively, and obtain a sequence of arcs $\gamma_i$ such that $d(\gamma,\gamma_i) \le i$ and \[\gamma_i \in \mathcal{U}(\beta_i^{w_{i+1}},\alpha_{i-1}^{w_i}) \cup  \mathcal{U}(\alpha_i^{w_{i+1}},\beta_{i-1}^{w_i}).\] We want to find a suitable arc $\gamma_{i+1}$; there are two possible cases. Suppose first that $\gamma_i\in \mathcal{U}(\beta_i^{w_{i+1}},\alpha_{i-1}^{w_i})$. Then, there exists \[\gamma_{i+1}\in \mathcal{U}(\beta_i^{w_{i+1}}, \alpha_{i+1}^{w_{i+2}}) \cup \mathcal{U}(\alpha_{i+1}^{w_{i+2}}, \alpha_{i-1}^{w_i}) \] with $d(\gamma_i,\gamma_{i+1})\le 1$. Once again, if $\gamma_{i+1}\in\mathcal{U}(\alpha_{i+1}^{w_{i+2}}, \alpha_{i-1}^{w_i})$, then  \[d(\gamma, \{\alpha,\beta\}) \leq d(\gamma, \gamma_i)+ d(\gamma_i,\gamma_{i+1}) + d(\gamma_{i+1}, \alpha) \leq i + 2,\] as  $d(\alpha_{i+1}, \alpha)=d(\alpha_{i-1}, \alpha)= 1$.  Otherwise,  $\gamma_{i+1} \in \mathcal{U}(\beta_i^{w_{i+1}}, \alpha_{i+1}^{w_{i+2}})$ and the induction process continues.

The case where $\gamma_i \in \mathcal{U}(\alpha_i^{w_{i+1}},\beta_{i-1}^{w_i})$ is analogous, and follows using the same argument as in the previous case, using $\beta$ instead of $\alpha$. 

After $k-1$ steps, we have that either $d(\gamma,\{\alpha, \beta\})\leq k$,  or else there is $\gamma_{k-1}$ with $d(\gamma,\gamma_{k-1})\leq k-1$ and such that 
 \[\gamma_{k-1}\in \mathcal{U}(\beta_{k-1}^{w_{k}},\alpha_{k-2}^{w_{k-1}})\cup  \mathcal{U}(\alpha_{k-1}^{w_{k}},\beta_{k-2}^{w_{k-1}}).\]
If $\gamma_{k-1}$ is in the first term of the union, there is
\[
\gamma' \in  \mathcal{U}(\beta_{k-1}^{w_{k}},\alpha_{k-1}^{w_{k-1}}) \cup \mathcal{U}(\alpha_{k-1}^{w_{k-1}},\alpha_{k-2}^{w_{k-1}})
\] with $i(\gamma',\gamma_{k-1}) = 0$. 
If $\gamma' \in \mathcal{U}(\alpha_{k-1}^{w_{k-1}},\alpha_{k-2}^{w_{k-1}})$, then $d(\gamma,\alpha)\leq k+1$. Otherwise,  $\gamma'\in  \mathcal{U}(\beta_{k-1}^{w_{k}},\alpha_{k-1}^{w_{k-1}})$ with $d(\gamma,\gamma')\leq k$. The second case is analogous, and produces either  $d(\gamma,\beta)\leq k+1$ or $\gamma'\in  \mathcal{U}(\alpha_{k-1}^{w_{k}},\beta_{k-1}^{w_{k-1}})$, proving the Claim and thus the Lemma also. 
\end{proof}

We are finally in a position for proving our result about the hyperbolicity of matching arc complexes: 

\begin{proof}[Proof of Theorem \ref{main:hyp}]
The result follows from the combination of Lemmas \ref{lem:guessing}, \ref{lem:thin} and \ref{lem:diam}
\end{proof}

\end{document}